\documentclass[12pt]{article}
\usepackage{e-jc}
\usepackage{amsthm,amssymb,amsfonts,amsmath,enumerate,graphicx,color,tikz}
\usepackage{mathtools}
\usepackage{mathrsfs}

\usepackage{caption}
\usepackage{subcaption}

\usepackage[utf8]{inputenc}

\usepackage{hyperref}
	\hypersetup{
	colorlinks=true,
	linkcolor={blue},
	citecolor={blue}
	}

\newcommand{\Cc}{\mathcal{C}}

\newcommand{\local}{$\text{}^{\ast}$local }
\newcommand{\maxideal}{\mathbf{m}}
\newcommand{\maximal}{${}^{\ast}$maximal }
\newcommand{\Pb}{\mathbf{P}}
\newcommand{\Ib}{\mathbf{I}}

\newcommand{\Z}{\mathbb{Z}}

\newcommand{\R}{\mathbb{R}}
\newcommand{\Q}{\mathbb{Q}}

\newcommand{\des}{\mathrm{des}}

\newcommand{\cone}{\mathrm{cone}}
\newcommand{\conv}{\mathrm{conv}}

\renewcommand{\phi}{\varphi}
\renewcommand{\emptyset}{\varnothing}

\def\b{{\boldsymbol b}}
\def\a{{\boldsymbol a}}
\def\c{{\boldsymbol c}}
\def\d{{\boldsymbol d}}
\def\e{{\boldsymbol e}}
\def\f{{\boldsymbol f}}
\def\g{{\boldsymbol g}}

\def\p{{\boldsymbol p}}
\def\s{{\boldsymbol s}}
\def\t{{\boldsymbol t}}
\def\u{{\boldsymbol u}}
\def\v{{\boldsymbol v}}
\def\x{{\boldsymbol x}}

\newcommand\rem[1]{}

\rem{

}
\newcommand\commentout[1]{}

\newcommand{\Znn}{\Z_{\geq 0}}
\newcommand{\Zpos}{\Z_{\geq 1}}

\newcommand{\Pbn}{\Pb_n^{(\s)}}
\newcommand{\Ibn}{\Ib_n^{(\s)}}

\newcommand{\asc}{\operatorname{asc}}
\newcommand{\Asc}{\operatorname{Asc}}

\def\Rc{\mathcal{R}}
\newcommand{\Ext}{\operatorname{Ext}}
\newcommand{\soc}{\operatorname{soc}}

\newcommand{\flocomment}[1]{ {\textcolor{violet}{ #1  --Florian}}} 
\newcommand{\colonelcomment}[1]{ {\textcolor{red}{ #1  --McCabe}}} 

\rem{
\newtheorem{theorem}{Theorem}[section]
\newtheorem{corollary}[theorem]{Corollary}

\newtheorem{lemma}[theorem]{Lemma}
\newtheorem{conjecture}[theorem]{Conjecture}

\theoremstyle{remark}

\newtheorem{remark}[theorem]{Remark}
\theoremstyle{definition}
\newtheorem{definition}[theorem]{Definition}

}

\dateline{Mar 28, 2019}{Aug 13, 2020}{TBD}

\Copyright{The authors. Released under the CC BY license (International 4.0).}

\title{Level algebras and $\s$-lecture hall polytopes}

\author{Florian Kohl\thanks{Supported by Academy of
Finland project number 324921.}\\
\small Department of Mathematics and Systems Analysis\\[-0.8ex]
\small Aalto University\\[-0.8ex] 
\small Espoo, Finland.\\
\small\tt florian.kohl@aalto.fi\\
\and
McCabe Olsen\\
\small Department of Mathematics\\[-0.8ex]
\small Rose-Hulman Institute of Technology\\[-0.8ex]
\small Terre Haute, IN 47803, USA\\
\small\tt olsen@rose-hulman.edu}

\begin{document}



\maketitle


\begin{abstract}
Given a family of lattice polytopes, a common endeavor in Ehrhart theory is the classification of those polytopes in the family that are Gorenstein, or more generally level. In this article, we consider these questions for $\s$-lecture hall polytopes, which are a family of simplices arising from $\mathbf {s}$-lecture hall partitions. In particular, we provide concrete classifications for both of these properties purely in terms of $\s$-inversion sequences. Moreover, for a large subfamily of $\s$-lecture hall polytopes, we provide a more geometric classification of the Gorenstein property in terms of  its tangent cones. We then show how one can use the classification of level $\s$-lecture hall polytopes to construct infinite families of level $\s$-lecture hall polytopes, and to describe level $\s$-lecture hall polytopes in small dimensions.\end{abstract}


\section{Introduction}
\label{sec:Intro}

Let $P\subset \R^n$ be a convex lattice polytope. 
It is a common question in Ehrhart theory to determine whether $P$ is a \emph{Gorenstein} polytope, that is, whether the associated semigroup algebra of $P$ is Gorenstein. 
\commentout{The precise definitions of semigroup algebras and the Gorenstein property can be found in later sections along with all other technical definitions of properties mentioned in this introduction. }
Gorenstein polytopes are also of interest within geometric combinatorics, as Gorenstein polytopes can be characterized in purely geometric terms. Particularly, $P$ is Gorenstein if and only  there is some integer dilate $cP$ which is a reflexive polytope \cite{DeNegriHibi}. 
Likewise, the Gorenstein property is equivalent to the interesting enumerative property that the  Ehrhart $h^\ast$-polynomial of $P$ has palindromic coefficients \cite{Stanley-HilbertFunctions}. 
Gorenstein polytopes are also of interest in algebraic geometry for a variety 
of reasons, including connections to mirror symmetry (see, e.g., \cite{Batyrev} and \cite[Section 8.3]{CoxLittleSchenck}). 
Roughly speaking, a pair of reflexive lattice polytopes gives rise to a mirror pair of Calabi--Yau manifolds. 
We recommend \cite{CoxMirrorSurvey} for an excellent survey article about reflexive polytopes and their connection to mirror symmetry.
Subsequently, classifications of the Gorenstein property have been extensively studied and are known for many families including order polytopes \cite{Stanley-PosetPolytopes,Hibi-LatticesAndStraighteningLaws}, twinned poset polytopes \cite{HibiMatsuda-twinned}, and $r$-stable $(n,k)$-hypersimplices \cite{HibiSolus}.

Gorenstein algebras are intimately related to level algebras. We say that $P$ is a \emph{level} polytope if its associated semigroup algebra is a \emph{level algebra}, a generalization of Gorenstein algebras.
Classifying level polytopes has not been studied to nearly the same degree as detecting the Gorenstein property (see, e.g., \cite{EneHerzogHibiSaeedi,HigashitaniYanagawa-Level}).
However, in addition to the independent interest in level algebras, if $P$ is level, we obtain nontrivial inequalities on the coefficients of the Ehrhart $h^\ast$-polynomial, which are not satisfied for general lattice polytopes (see, e.g., \cite{StanleyGreenBook}). 

One family of well-studied polytopes  are the {$\s$-lecture hall polytopes}. For a given $\s\in\Zpos^n$, the \emph{$\s$-lecture hall polytope} is the simplex defined by
	\[
	\Pbn\coloneqq\left\lbrace \boldsymbol \lambda\in \R^n \, : \, 0\leq \frac{\lambda_1}{s_1}\leq\frac{\lambda_2}{s_2}\leq \cdots\leq \frac{\lambda_n}{s_n}\leq 1 \right\rbrace.
	\]
In the literature, $\s$-lecture hall polytopes are also sometimes called \emph{$\s$-lecture hall simplices}. 
These polytopes arise from the extensively investigated \emph{$\s$-lecture hall partitions}, introduced by Bousquet-M\'{e}lou and Eriksson \cite{BME-LHP1,BME-LHP2}. To quote Savage and Schuster from \cite{SavageSchuster}: ``Since their discovery, $\s$-lecture hall partitions and their generalizations have emerged as fundamental tools for interpreting classical partition identities and for discovering new ones.''
Though many algebraic and geometric properties of $\s$-lecture hall polytopes are known (see, e.g., \cite{Savage-LHP-Survey}), there is not an explicit full characterization of the Gorenstein property and there are no known levelness results to date.

Our focus is to determine a classification of the Gorenstein and level properties in $\s$-lecture hall polytopes.
In particular, we provide a full characterization for the Gorenstein property. We also give another more geometric characterization in the case that $\s$ has at least one index $i$, $1< i\leq n$, such that $\gcd(s_{i-1},s_{i})=1$. 
These main results on the Gorenstein property are as follows:   	 
\begin{theorem}
\label{thm:NewGorensteinChar}
Let $\s\rem{=(s_1,s_2,\ldots,s_n)}\in \Z^{n}_{\geq 1}$. Then $\Pbn$ is Gorenstein if and only if there exists $\c\in\Z^{n+1}$  satisfying $c_1=1$,
	\[
	c_js_{j-1}=c_{j-1}s_j+\gcd(s_{j-1},s_{j})
	\]
for $j>1$,  and 
\[
c_{n+1}s_n = 1 + c_n.
\]

\end{theorem}

The next result is not as general, but it guarantees that, under the condition that $\gcd(s_{i-1},s_i)=1$ for some $1<i \leq n$, if the two vertex cones of $\Pbn$ at $(0,0,\dots,0)$ and at $(s_1,s_2,\dots, s_n)$ are Gorenstein, then $\Pbn$ is Gorenstein. While we will use this geometric perspective in the proof, we give an equivalent reformulation highlighting that the Gorenstein condition is an explicit condition on $\s$.

\begin{theorem}\commentout{[see Theorem \ref{GorensteinChar}]}
\label{GorensteinChar}
Let $\s\in \Z_{\geq 1}^n$ be such that $\gcd(s_{i-1},s_i)=1$ for some $1<i\leq n$. The polytope $\Pbn$ is Gorenstein if and only if for all $j \geq 2$
\begin{equation}
\label{eq:GorChar}
\rem{c_j \coloneqq }\frac{\gcd(s_{j-1},s_j)}{s_{j-1}} + \sum_{k=1}^{j-1} \frac{\gcd(s_{k-1},s_k)}{s_{k-1}s_k} \quad \text{and } \quad \rem{d_j \coloneqq }\frac{\gcd(s_{n-j+2},s_j)}{s_{n-j+1}} + \sum_{k=1}^{j-1} \frac{\gcd(s_{n-k+2},s_{n-k+1})}{s_{n-k+2}s_{n-k+1}}
\end{equation}
are integers where $s_0 = s_{n+1}=1$.
\rem{, and define $\overleftarrow{\s}=(\overleftarrow{s_1},\ldots,\overleftarrow{s_n})\coloneqq(s_n,s_{n-1},\ldots,s_1)$. Then $\Pbn$ is Gorenstein if and only if there exist $\c,\d\in\Z^n$ with $c_1=d_1=1$ satisfying
	\[
	c_js_{j-1}=c_{j-1}s_j+\gcd(s_{j-1},s_{j})
	\]
and 
	\[
	d_j\overleftarrow{s_{j-1}}=d_{j-1}\overleftarrow{s_j}+\gcd(\overleftarrow{s_{j-1}},\overleftarrow{s_{j}})
	\]	
for $j>1$.}
\end{theorem}

\begin{remark}
It is straightforward to show that $\Pbn$ is Gorenstein if and only if $P_{n+1}^{(1,\s)}$ is Gorenstein, as $P_{n+1}^{(1,\s)}$ is the lattice pyramid over $\Pbn$. Since $P_{n+1}^{(1,\s)}$ satisfies the conditions of Theorem~\ref{GorensteinChar}, one can apply Theorem~\ref{GorensteinChar} to \emph{any} $\s$-lecture hall polytope.
\end{remark}

Moreover, we provide a characterization for levelness.
For a given $\s$, let $\Ibn\coloneqq\{e\in\Znn^n \, : \, 0\leq e_i<s_i\}$ be the set of \emph{$\s$-inversion sequences}. Given $\e\in \Ibn$, let $\asc(\e)$ be the \emph{ascent number} of $\e$ and let  $\Ib_{n,k}^{(\s)}$ denote the set of inversion sequences with ascent number $k$.
Furthermore,  for two inversion sequences $\e_1,\e_2\in\Ibn$, we say that $\e_1+\e_2$ is the inversion sequence formed by componentwise addition where the $i$th component is considered modulo $s_i$. 
These notions will be defined more thoroughly later sections.
Our characterization is the following theorem: 

\begin{theorem}
\label{THM:LEVEL}
Let $\s=(s_1,s_2,\ldots, s_n)$ and let $r=\max\{\asc(\e) \, : \, \e\in \Ibn \}$. Then $\Pbn$ is level if and only if for any $\e\in\Ib_{n,k}^{(\s)}$ with $1\leq k <r$ there exists some $\e'\in \Ib_{n,1}^{(\s)}$ such that $(\e+\e')\in \Ib_{n,k+1}^{(\s)}$.
\end{theorem}

The structure of this manuscript is as follows. In Section \ref{sec:Background}, we provide all necessary background, definitions, notation, and terminology.
The focus of Section \ref{sec:Gorenstein} is proving the Gorenstein classifications.
In Section \ref{sec:Level}, we prove the characterization of the level property. 
We conclude in Section \ref{sec:Concluding} with some potential ways to improve and extend these results and other future directions.\\


\section{Background}\label{sec:Background}
In this section, we provide the necessary terminology and background literature for our results. Specifically, we review lattice polytopes and Ehrhart theory, Gorenstein algebras and level algebras, and the polyhedral geometry of $\s$-lecture hall partitions. Subsequently, some or all of these subsections may be safely skipped by the experts.

\subsection{Lattice polytopes and Ehrhart theory}\label{sec:Background:Ehrhart}
A \emph{polytope} $P\subset \R^n$ is the convex hull of finitely many points in $\R^n$, i.e,
	\[
	P=\conv \{\v_1,\ldots, \v_r\}: = \left\{ \sum_{i=1}^r \lambda_i \v_i \colon \, \lambda_i \geq 0, \sum \lambda_i = 1, \v_i \in \R^n \right\}  .
	\]
The inclusion-minimal set $V$ such that $P = \conv \{ \v \colon \v \in V\}$ is called the \emph{vertex set}, and its elements are called the \emph{vertices of $P$}. The polytope $P$ is a \emph{lattice} polytope if $V \subset \Z^n$.
The \emph{Ehrhart polynomial $i(P,t) \colon \Z_{\geq 1} \to \Z_{\geq 1}$} of $P$ is the function 
	\[
	i(P,t)\coloneqq\#(tP\cap \Z^n)
	\]
which agrees with polynomial in the variable $t$ of degree $d=\dim(P)$ by a result of Ehrhart \cite{Ehrhart}. The \emph{Ehrhart series} of $P$ is the rational generating function
	\[
	1+\sum_{t\geq 1}i(P,t) z^t=\frac{h^\ast(P,z)}{(1-z)^{d+1}},
	\]
where the numerator is the polynomial 
	\[
	h^\ast(P,z)=\sum_{j=0}^{d}h_j^\ast(P)z^j,
	\]
which we call the \emph{Ehrhart $h^\ast$-polynomial} of $P$. The coefficient vector\\ $h^{\ast}(P)=(h^{\ast}_0(P),h^{\ast}_1(P),\dots,h^{\ast}_d(P))$ is known as the \emph{$h^{\ast}$-vector}.
If the polytope is clear from context, we will simplify our notation to 	$(h^{\ast}_0,h^{\ast}_1,\dots,h^{\ast}_d)$. 
By a result of Stanley \cite{Stanley-NonNeg}, we know that $h_j^\ast(P)\in \Znn$ for all $j$. 
Many additional properties are known about Ehrhart $h^\ast$-polynomials (see, e.g, \cite{BeckRobins-CCD,Hibi-ACCP}). Classifying the set of $h^{\ast}$-vectors is one of the most important open problems in Ehrhart theory. Therefore, inequalities for the coefficients are of special interest, see \cite{Stapledon1, Stapledon2, HibiInequalities, StanleyInequalities}. Hofscheier, Katth\"an, and Nill proved a structural result about $h^{\ast}$-vectors, see \cite[Theorem 3.1]{Spanning}, where they showed that if the integer points of a lattice polytope $P$ span the integer lattice, then $h^{\ast}(P,z)$ cannot have internal zeros. There are even some universal inequalities for $h^{\ast}$-vectors, i.e., inequalities independent of the degree and the dimension of the polytope, see \cite{BallettiHigashitani}.  

 Given a lattice polytope $P$ with vertex set $V(P)$, define the \emph{cone over} $P$ to be
	\[
	\cone(P): ={\rm span}_{\R_{\geq 0}}\{(\v,1) \, : \, \v\in V(P)\}\subset \R^n\times \R.
	\]
Let $\v$ be a vertex of  $P$. The \emph{vertex cone of $P$ at $\v$}  is defined as
\[
T_{\v}(P) \coloneqq \{\v + \lambda (\x - \v) \colon \x \in P, \lambda \geq 0\}.
\]
The vertex cone $T_{\v}(P)$ is also known as the \emph{tangent cone} of $P$ at $\v$.
Let $F$ be a facet of a lattice polytope $P$ ($\cone(P)$, respectively) corresponding to $\langle \a_F, \x \rangle = b_F$, where $a_F$ is primitive.  If $|\langle a_F, \x \rangle - b| = d$, then we say that $\x \in P$ (or $x \in \cone(P)$, respectively) has \emph{lattice distance $d$} to $F$.

Let $k$ be an algebraically closed field of characteristic zero. We define the \emph{affine semigroup algebra} of $P$ to be 
	\[
	k[P]\coloneqq k[\cone(P)\cap \Z^{n+1}]=k[\x^{\p}\cdot y^m \, : \, (\p,m)\in  \cone(P)\cap \Z^{n+1}]\subset k[x_1^{\pm 1},\ldots, x_n^{\pm 1},y].
	\]	

This algebra is known to be a finitely generated, local $k$-algebra with a natural $\Znn$-grading arising from the $y$-degree (see, e.g., \cite{MillerSturmfels-CCA}). Moreover, $k[P]$ is Cohen-Macaulay \cite{Hochster}. Given the observation that the lattice points $(\p,m)\in \cone(P)\cap \Z^{n+1}$ in the cone are in clear bijection with elements in $mP\cap \Z^n$, the Ehrhart polynomial is the Hilbert function for the algebra $k[P]$.

We say that $P$ satisfies the \emph{integer decomposition property} (or IDP for short) if for any $\x\in t P\cap \Z^n$, there exist $t$ lattice points $\{\p_1,\p_2,\ldots,\p_t\}\in P\cap \Z^n$ such that $\p_1+\p_2+\cdots +\p_t=\x$. 
Equivalently, $P$ satisfies the IDP if the semigroup algebra $k[P]$ is generated entirely in degree 1.

Suppose that $P$ is a simplex and has vertex set $\{\v_0,\cdots,\v_{d}\}$. The \emph{(half-open) fundamental parallelepiped} of $P$ is the bounded region of $\cone(P)$ defined as
	\[
	\Pi_P\coloneqq\left\lbrace \sum_{i=0}^{d} \eta_i(\v_i,1) \, : \, 0\leq \eta_i <1 \right\rbrace\subset \cone(P).
	\] 
For simplices, we can use the fundamental parallelepiped to compute the Ehrhart $h^\ast$-polynomial. In particular, the coefficients are given by
	\[
	h^\ast_i(P)=\#\left\{\x \in \Pi_P\cap \Z^{n+1} \colon \, \x = (x_1,\dots, x_n,i)\right\}, 
	\]
that is, the number of lattice points at height $i$ in $\Pi_P$.
For more details and exposition, the reader should consult \cite{BeckRobins-CCD}. 		

\subsection{Gorenstein algebras and level algebras}\label{sec:Backgroup:Gorenstein}
We now provide a brief review of Gorenstein and level algebras. Since we will only be concerned with semigroup algebras of polytopes, we will restrict ourselves to this case. \rem{in general, as well as in the special case of semigroup algebras of polytopes. }
For additional details and expositions, the reader should consult \cite{BrunsHerzog,StanleyGreenBook} as references. 

\rem{We begin by defining a graded version of locality:
\begin{definition}[{\cite[Definition 6.15]{BrunsGubeladze}}]
\label{def:ch1StarLocal}
Let $\Rc$ be a $\Z_{\geq 0}$-graded ring. Then we say that $\Rc$ is \emph{\local with \maximal ideal $\maxideal$} if the homogeneous non-units of $\Rc$ generate the proper ideal $\maxideal$.
\end{definition}

Let $k$ be an algebraically closed field of characteristic zero. Let $\mathcal{R}=\bigoplus_{i\in\Z} \mathcal{R}_i$ be a finitely generated $\Z$-graded $k$-algebra of Krull dimension $d$. Suppose that $\Rc$ is \local and Cohen-Macaulay. 
The \emph{canonical module}, $\omega_\Rc$, of $\Rc$ is the unique module (up to isomorphism)  such that
$\Ext_\Rc^d(k,\omega_\Rc)=k$ and $\Ext_\Rc^i(k,\omega_\Rc)=0$ when $i\neq d$. We say that $\Rc$ is \emph{Gorenstein} if $\omega_\Rc\cong\Rc$ as an $\Rc$-module, or equivalently if $\omega_\Rc$ is generated by a single element.

One generalization of the Gorenstein property which is also of interest is the \emph{level} property.
We say that $\Rc$ is \emph{level} if $\omega_\Rc$ is generated by elements of the same degree, that is, $\omega_\Rc$ has minimal generating set $\{\sigma_1,\ldots,\sigma_j\}$ such that $\deg(\sigma_1)=\deg(\sigma_2)=\cdots=\deg(\sigma_j)$. 
An equivalent formulation of the level property is often more fruitful for computational purposes. Recall for any $\Rc$-module $M$, the \emph{socle} of $M$ is $\soc(M)\coloneqq\{u\in M \, : \, R_+u=0\}$ where $R_+$ is the ideal generated by the homogeneous non-units of $\Rc$. \rem{unique \maximal ideal of $\Rc$.}
It is equivalent to say that $\Rc$ is level if for any homogeneous system of parameters $\theta_1,\ldots,\theta_d$ of $\Rc$, all the elements of the graded vector space $\soc(\Rc/(\theta_1,\ldots,\theta_d))$ are of the same degree. Such a homogeneous system of parameters always exists, since $\Rc$ is Cohen--Macaulay, see \cite[Chapter III, Proposition 3.2]{StanleyGreenBook}.

}

In commutative algebra, the Gorenstein property of a graded $k$-algebra $\Rc$ is often defined in terms of the canonical module $\omega_{\Rc}$. In the case of a semigroup algebra $k[P]$ of a lattice polytope $P$, Stanley \cite{Stanley-HilbertFunctions} explicitly describes the canonical module as
	\[
	\omega_{k[P]}=k[\cone(P)^\circ\cap\Z^{n+1}]
	\]
where $\cone(P)^\circ$ denotes the relative interior of the cone. We say that $k[P]$ is \emph{Gorenstein} if there exists $\c\in\Z^{n+1}$ such that 
	\[
	\c+(\cone(P)\cap\Z^{n+1})=\cone(P)^\circ\cap\Z^{n+1},
	\]	
and $\c$ is called the \emph{Gorenstein point of $\cone(P)$}. Equivalently, $k[P]$ is Gorenstein if and only if there is a $\c\in\Z^{n+1}$ having lattice distance $1$ to all facets of $\cone(P)$, see~\cite[Theorem 6.32]{BrunsGubeladze}. Moreover, note that $P$ is Gorenstein if and only if $h^\ast(P,z)$ is a palindromic polynomial, see \cite[Theorem 4.4]{Stanley-HilbertFunctions}.

One generalization of the Gorenstein property which is also of interest is the \emph{level} property.
We say that $k[P]$ is \emph{level} if $\omega_{k[P]}$ is generated by elements of the same degree, that is, $\omega_{k[P]}$ has minimal generating set $\{\sigma_1,\ldots,\sigma_j\}$ such that $\deg(\sigma_1)=\deg(\sigma_2)=\cdots=\deg(\sigma_j)$. 
An equivalent formulation of the level property is often more fruitful for computational purposes. Recall for any ${k[P]}$-module $M$, the \emph{socle} of $M$ is $\soc(M)\coloneqq\{u\in M \, : \, R_+u=0\}$ where $R_+$ is the ideal generated by the homogeneous non-units of ${k[P]}$. \rem{unique \maximal ideal of ${k[P]}$.}
It is equivalent to say that ${k[P]}$ is level if for any homogeneous system of parameters $\theta_1,\ldots,\theta_d$ of ${k[P]}$, all the elements of the graded vector space $\soc({k[P]}/(\theta_1,\ldots,\theta_d))$ are of the same degree, see \cite[Chapter III, Proposition 3.2]{StanleyGreenBook}. \rem{Such a homogeneous system of parameters always exists, since ${k[P]}$ is Cohen--Macaulay.}

We can also provide a more concrete description of  the level property. We say $k[P]$ is \emph{level} if there exists some finite collection $\c_1,\ldots,\c_m\in \Z^{n+1}$ where
	\[
	\sum_{i=1}^m (\c_i+(\cone(P)\cap\Z^{n+1}))=\cone(P)^\circ\cap\Z^{n+1},
	\]	
and the additional restriction that $c_{1_{n+1}}=c_{2_{n+1}}=\cdots=c_{m_{n+1}}$. For a lattice polytope $P$, we say that $P$ is \emph{Gorenstein} (respectively, \emph{level}) if $k[P]$ is Gorenstein (respectively, level).

\subsection{Polyhedral geometry of $\s$-lecture hall partitions}\label{sec:Background:LectureHall}
In this subsection, we briefly recall relevant properties and results on $\s$-lecture hall cones and $\s$-lecture hall polytopes. For a more in-depth overview of some of these results and many others, the reader should consult the excellent survey of Savage \cite{Savage-LHP-Survey}.

Let $\s=(s_1,s_2,\ldots,s_n)\in \Zpos^n$ be a sequence. Given any $\s$-sequence, define the \emph{$\s$-lecture hall partitions} to be the set 
	\[
	L_n^{(\s)}\coloneqq\left\lbrace \boldsymbol \lambda\in \Z^n \, : \, 0\leq \frac{\lambda_1}{s_1}\leq\frac{\lambda_2}{s_2}\leq \cdots\leq \frac{\lambda_n}{s_n} \right\rbrace.
	\] 	 
We can associate to the set of $\s$-lecture hall partitions several discrete geometric objects, in particular, the \emph{$\s$-lecture hall polytope} and the \emph{$\s$-lecture hall cone}.
For a given $\s$, the  \emph{$\s$-lecture hall polytope}  is defined 
	\[
	\begin{array}{rcl}
	\Pbn & \coloneqq &  \displaystyle \left\lbrace \boldsymbol\lambda\in \R^n \, : \, 0\leq \frac{\lambda_1}{s_1}\leq\frac{\lambda_2}{s_2}\leq \cdots\leq \frac{\lambda_n}{s_n}\leq 1 \right\rbrace\\
	& = & \conv\{(0,\ldots,0), (0,\ldots,0,s_i,s_{i+1},\ldots,s_n) \ \mbox{for } 1\leq i\leq n\}.
	\end{array}
	\]

The  Ehrhart $h^\ast$-polynomials of $\Pbn$  have been completely classified. Given $\s$, the set of \emph{$\s$-inversion sequences} is defined as $\Ibn\coloneqq\{e\in\Znn^n \, : \, 0\leq e_i<s_i\}$. For a given $\e\in\Ibn$, the \emph{ascent set} of $\e$ is
	\[
	\Asc(\e)\coloneqq\left\lbrace i\in \{0,1,\ldots,n-1\} \, : \, \frac{e_i}{s_i}<\frac{e_{i+1}}{s_{i+1}} \right\rbrace,
	\]
with the conventions $s_0=1$, $e_0=0$, and $\asc(\e)\coloneqq\#\Asc(\e)$. 
With these definitions, we can give the explicit formulation for the Ehrhart $h^\ast$-polynomials.	
\begin{theorem}[{\cite[Theorem 8]{SavageSchuster}}]
For a given $\s\in\Zpos^n$, 
	\[
	h^\ast(\Pbn,z)=\sum_{\e\in\Ibn}z^{\asc(e)}.
	\]
\end{theorem}

The polynomials $h^\ast(\Pbn,z)$ are known as the \emph{$\s$-Eulerian polynomials}, because they generalize the classical Eulerian polynomials.
Let $\mathfrak{S}_n$ denote the symmetric group of $[n]$.
Given $\pi=\pi_1\pi_2\cdots\pi_n \in \mathfrak{S}_n$, recall that the \emph{descent statistic} of $\pi$ is $\des(\pi)=\#\{i\in [n-1] \ : \pi_i>\pi_{i+1}\}$.
This statistic on permutations gives rise to one definition of the classical \emph{Eulerian polynomial}
	\[
	A_n(z)\coloneqq\sum_{\pi\in\mathfrak{S}_n} z^{\des(\pi)}.
	\]
In the special case of $\s=(1,2,\ldots,n)$, we have
	\[
	h^\ast( \Pb_n^{(1,2,\ldots,n)},z)= \sum_{\e\in\Ib_n^{(1,2,\ldots,n)}}z^{\asc(e)}=\sum_{\pi\in\mathfrak{S}_n} z^{\des(\pi)}=A_n(z).
	\]	
The $\s$-Eulerian polynomials are known to be real-rooted and, hence, unimodal \cite{SavageVisontai}.	

In recent years, $\s$-lecture hall polytopes have been the subject of much additional study (see, e.g,.  \cite{HibiOlsenTsuchiya-SelfDual,PensylSavage-Wreath,PensylSavage-Rational,SavageViswanathan-1/kEulerian}). Of particular interest are algebraic and geometric structural results such as Gorenstein and IDP.
The second author along with Hibi and Tsuchiya in \cite{HibiOlsenTsuchiya-LHPGorensteinIDP} provide some Gorenstein results in particular circumstances. 
Additionally, the following theorem for IDP holds.

\begin{theorem}[{\cite[Theorem 2.1]{BrandenSolus}}]
$\Pbn$ has the IDP.
\label{IDP}
\end{theorem}

A proof for the case of monotonic $\s$-sequences was given by the second author with Hibi and Tsuchiya in \cite{HibiOlsenTsuchiya-LHPGorensteinIDP} which  Br\"{a}nd\'{e}n and Solus \cite{BrandenSolus} show can be extended to any $\s$ when they prove that all \emph{$\s$-lecture hall order polytopes} have the IDP. Moreover, in \cite[Conjecture 5.4]{BraendenLeander} it is conjectured that for any $\s$, $\Pbn$ possesses a regular, unimodular triangulation.

For a given $\s$, the \emph{$\s$-lecture hall cone} is defined to be
	\[
	\Cc_n^{(\s)}\coloneqq\left\lbrace \boldsymbol\lambda\in \R^n \, : \, 0\leq \frac{\lambda_1}{s_1}\leq\frac{\lambda_2}{s_2}\leq \cdots\leq \frac{\lambda_n}{s_n} \right\rbrace,
	\]
and whose integer points are exactly the $\s$-lecture hall partitions. These objects are also related to $\s$-lecture hall polytopes in that $\Cc_n^{(\s)}$ arises as the vertex cone of $\Pbn$ at the origin $(0,\ldots,0)$. 
It is important to realize that $\Cc_n^{(\s)}$ is not the same object as $\cone(\Pbn)$. In fact, $\Cc_n^{(\s)}$ is the image of the map  $q_{\mathbf 0}~\colon~\cone(\Pbn)~\to~\R^n$, $(\x,h) \mapsto \x$, where $\x \in h \Pbn$.
The $\s$-lecture hall cones have been studied extensively (see, e.g., \cite{BeckEtAl-GorensteinLHC,BeckEtAl-TriangulationsLHC,Olsen-HilbertBases}) and the following Gorenstein results for the $\s$-lecture hall cones are particularly of interest for our purposes.

\begin{theorem}[{\cite[Corollary 2.6]{BeckEtAl-GorensteinLHC}},  {\cite[Proposition 5.4]{BME-LHP2}}] \label{conesGor}
For a positive integer sequence $\s$, the $\s$-lecture hall  cone $\Cc_n^{(\s)}$ is Gorenstein if and only if there exists some $\c \in \Z^n$ satisfying
	\[
	c_j s_{j-1}=c_{j-1}s_j+\gcd(s_{j-1},s_{j})
	\]
for $j>1$, with $c_1=1$. 	
\end{theorem}

Moreover, in the case of $\s$-sequences where $\gcd(s_{i-1},s_{i})=1$ holds for all $i$, we have a refinement of this theorem. 
We say that $\s$ is \emph{$\u$-generated} by a sequence $\u$ of positive integers if $s_2=u_1s_1-1$ and $s_{i+1}=u_is_i-s_{i-1}$ for $i>1$. 

\begin{theorem}[{\cite[Theorem 2.8]{BeckEtAl-GorensteinLHC}}, {\cite[Proposition 5.5]{BME-LHP2}}]\label{conesGorPairwise}
Let $\s=(s_1,\ldots,s_n)$ be a sequence of positive integers  such that $\gcd(s_{i-1},s_{i})=1$ for $1\leq i<n$. Then $\Cc_n^{(\s)}$ is Gorenstein if and only if $\s$ is $\u$-generated by some sequence $\u=(u_1,u_2,\ldots,u_{n-1})$ of positive integers. When such a sequence exists, the Gorenstein point $\c$ for $\Cc_n^{(\s)}$ is defined by $c_1=1$, $c_2=u_1$, and for $2\leq i<n$, $c_{i+1}=u_ic_i-c_{i-1}$.
\end{theorem}



\section{Gorenstein $\s$-lecture hall polytopes}
\label{sec:Gorenstein}
In this section, we will give a characterization of Gorenstein $\s$-lecture hall polytopes. 
To give such a classification, we will analyze the structure of $\cone(\Pbn)$.
The following lemma gives a halfspace inequality description of this cone:

\begin{lemma}
\label{lem:H-DescriptionConeOverPolytope}
With the notation from above, 
\begin{equation*}
\cone(\Pbn) = \left \{\boldsymbol\lambda \in \R^{n+1} \colon A\boldsymbol\lambda^t \geq \boldsymbol 0 \right\},
\end{equation*}
where
\begin{equation*}A=
\begin{pmatrix}
\frac{1}{s_1} & 0 &0& \dots & 0 \\
\frac{-1}{s_1} & \frac{1}{s_2}&0 & \dots &0\\
\vdots & \vdots & \vdots & \vdots & \vdots\\
0 & \dots & \frac{-1}{s_{n-1}} & \frac{1}{s_n} & 0 \\
0 & \dots & 0 & \frac{-1}{s_n} & 1\\
\end{pmatrix}.
\end{equation*}
Moreover, this cone is simplicial.
\end{lemma}
\begin{proof}
This directly follows from the  halfspace description of $\Pbn$. Assume that $\Pbn = \{\lambda \colon M\boldsymbol\lambda^t \geq \b\}$, where $\b = (0,0,\dots,0,1)^t$. Then on height $\lambda_{n+1}$, we have $M\boldsymbol\lambda^t \geq \lambda_{n+1}\b$. The statement now follows.  
\end{proof}

\begin{proof}[Proof of Theorem ~\ref{thm:NewGorensteinChar}]
Lemma~~\ref{lem:H-DescriptionConeOverPolytope} implies that $\cone\left(\Pbn\right) = \mathcal{C}^{(s_1,\dots,s_n,1)}_{n+1}$, i.e., $\cone\left(\Pbn\right)$ is an $\s$-lecture hall cone itself, which also appears implicitly in~\cite[Lemma 2.3]{LiuStanley}. Now the claim follows from\rem{ one can deduce Theorem~\ref{thm:NewGorensteinChar} as it is implicit from} Theorem~\ref{conesGor}.\rem{ and {\cite[Lemma 2.1]{LiuStanley}}.}
\end{proof}


In the interest of proving the alternative characterization given by Theorem~\ref{GorensteinChar}, we now recall a technical lemma necessary for the proof.

\begin{lemma}[{\cite[Lemma 2.5]{BeckEtAl-GorensteinLHC}}]
\label{GorLemma}
Let $\Cc=\{\boldsymbol\lambda\in\R^n \, : \, A\boldsymbol\lambda\geq \boldsymbol0\}$ be a full dimensional simplicial polyhedral cone where $A$ is a rational matrix and denote the rows of $A$ as linear functionals $\alpha^1,\ldots, \alpha^n$ on $\R^n$. For $j=1,\ldots,n$, let the projected lattice $\alpha^j(\Z^n)\subset \R$ be generated by the number $q_j\in\Q_{>0}$, so $\alpha^j(\Z^n)=q_j\Z$.
	\begin{enumerate}
	\item The cone $\Cc$ is Gorenstein if and only if there exists $\c\in\Z^n$ such that $\alpha^j(\c)=q_j$ for all $j=1,\ldots,n$.
	\item Define a point $\boldsymbol{\tilde{c}}\in \Cc\cap \Q^n$ by $\alpha^j(\boldsymbol{\tilde{c}})=q_j$ for all $j=1,\ldots,n$. Then $\Cc$ is Gorenstein if and only if $\boldsymbol{\tilde{c}}\in\Z^n$.
	\end{enumerate}
\end{lemma}

We now provide a proof of  Theorem~\ref{GorensteinChar}.

\commentout{
\begin{theorem}\label{GorensteinChar}
Let $\s=(s_1,s_2,\ldots,s_n)$ be a sequence such that there exists some $1\leq i< n$ such that $\gcd(s_{i},s_{i+1})=1$.  Then $\Pbn$ is Gorenstein if and only if there exists $\c,\d\in\Z^n$ satisfying
	\begin{equation}\label{eq:Crecursion}
	c_js_{j-1}=c_{j-1}s_j+\gcd(s_{j-1},s_{j})
	\end{equation}
and 
	\begin{equation}\label{eq:Drecursion}
	d_j\overleftarrow{s_{j-1}}=d_{j-1}\overleftarrow{s_j}+\gcd(\overleftarrow{s_{j-1}},\overleftarrow{s_{j}})
	\end{equation}	
for $j>1$ with $c_1=d_1=1$.	
\end{theorem}}

\begin{proof}[Proof of Theorem \ref{GorensteinChar}]
Let us define $\overleftarrow{\s}\rem{\coloneqq(\overleftarrow{s_1},\ldots,\overleftarrow{s_n})}\coloneqq(s_n,s_{n-1},\ldots,s_1)$.
We will verify the following two claims:
\begin{enumerate}
\item[(i)] If $P$ is \emph{any} Gorenstein polytope, then all of its vertex cones are Gorenstein as well. The terms in~(\ref{eq:GorChar}) actually correspond to coordinates of the Gorenstein points of the two vertex cones at $\mathbf 0$ and $\s$.
\item[(ii)] If the vertex cones of $\Pbn$ at $\mathbf 0 $ and $\s$ are both Gorenstein, then $\Pbn$ is Gorenstein. This direction does \emph{not} hold for general lattice polytopes. 
\end{enumerate}

\rem{
Let us introduce $\c \coloneqq (c_1,\dots, c_n)$ and $\d \coloneqq (d_1,\dots,d_n)$ by setting
\begin{equation*}
c_j \coloneqq \frac{\gcd(s_{j-1},s_j)}{s_{j-1}} + \sum_{k=1}^{j-1} \frac{\gcd(s_{k-1},s_k)}{s_{k-1}s_k} \quad \text{and } \quad d_j \coloneqq  \frac{\gcd(s_{n-j+2},s_j)}{s_{n-j+1}} + \sum_{k=1}^{j-1} \frac{\gcd(s_{n-k+2},s_{n-k+1})}{s_{n-k+2}s_{n-k+1}}.
\end{equation*}
Equivalently, $\c$ and $\d$ satisfy the recursions 
	\[
	c_js_{j-1}=c_{j-1}s_j+\gcd(s_{j-1},s_{j})
	\]
and 
	\[
	d_j\overleftarrow{s_{j-1}}=d_{j-1}\overleftarrow{s_j}+\gcd(\overleftarrow{s_{j-1}},\overleftarrow{s_{j}})
	\]	
for $j>1$, where $c_1=d_1=1$.  A quick computations shows that $\c$ in $\d$ satisfy the recursions the terms in ~(\ref{eq:GorChar}) are integers if and only if  if there exist $\c,\d\in\Z^n$ with $c_1=d_1=1$ satisfying
	\[
	c_js_{j-1}=c_{j-1}s_j+\gcd(s_{j-1},s_{j})
	\]
and 
	\[
	d_j\overleftarrow{s_{j-1}}=d_{j-1}\overleftarrow{s_j}+\gcd(\overleftarrow{s_{j-1}},\overleftarrow{s_{j}})
	\]	
for $j>1$. This reformulation has the advantage that there is a geometric interpretation of $\c$ and $\d$ as they will arise as Gorenstein points of tangent cones.

, and define $\overleftarrow{\s}=(\overleftarrow{s_1},\ldots,\overleftarrow{s_n})\coloneqq(s_n,s_{n-1},\ldots,s_1)$. Then $\Pbn$ is Gorenstein if and only if there exist $\c,\d\in\Z^n$ with $c_1=d_1=1$ satisfying
	\[
	c_js_{j-1}=c_{j-1}s_j+\gcd(s_{j-1},s_{j})
	\]
and 
	\[
	d_j\overleftarrow{s_{j-1}}=d_{j-1}\overleftarrow{s_j}+\gcd(\overleftarrow{s_{j-1}},\overleftarrow{s_{j}})
	\]	
for $j>1$.}


Let $\Pbn$ be Gorenstein with Gorenstein point $\b \in \cone (\Pbn)\cap \Z^{n+1}$. We will show that there are integer points $\c$ and $\d$ satisfying the recursions 
	\begin{equation}\label{eq:Crecursion}
	c_js_{j-1}=c_{j-1}s_j+\gcd(s_{j-1},s_{j})
	\end{equation}
and 
	\begin{equation}\label{eq:Drecursion}
	d_j\overleftarrow{s_{j-1}}=d_{j-1}\overleftarrow{s_j}+\gcd(\overleftarrow{s_{j-1}},\overleftarrow{s_{j}})
	\end{equation}	
for $j>1$, where $c_1=d_1=1$. Solving  $(\ref{eq:Crecursion})$ and $(\ref{eq:Drecursion})$ for $c_j$ and $d_j$ will then lead to the integrality conditions in  $(\ref{eq:GorChar})$. Since $\b$ is a Gorenstein point, it has lattice distance $1$ to all facets of $\cone(\Pbn)$ by~\cite[Theorem 6.32]{BrunsGubeladze}. For a vertex $\v$, the vertex cone $T_{\v}(\Pbn)$ can be related to $\cone(\Pbn)$ by considering the map $q_\v \colon \cone(\Pbn)\to \R^n$, $(\x,h) \mapsto \x - h\v$, where $\x \in h\Pbn$. It is straightforward to see that $T_{\v}(\Pbn) = \v + q_v(\cone(\Pbn))$, and $T_{\v}(\Pbn)$ is Gorenstein if and only if $q_\v(\cone(\Pbn))$ is Gorenstein. Moreover, as one can quickly verify, $q_\v$ preserves facet distances; if $F$ is a facet of $\cone(\Pbn)$ containing $(\v,1)$ and $(\x,h)$ has facet distance $d$ to $F$, then $q_\v(\x,h)$ has facet distance $d$ to $q_{\v}(F)$. Therefore, $q_\v(\b)$ has lattice distance $1$ to all facets of $q_v(\cone(\Pbn))$ and $q(\b)$ is a Gorenstein point and thus $T_{\v}(\Pbn)$ is Gorenstein for all. Hence, all vertex cones are Gorenstein.  
 
In particular, the vertex cone at the vertex $(0,0,\dots,0)$ is of the form
	\[
	0 \leq \frac{\lambda_1}{s_1} \leq \dots \leq \frac{\lambda_n}{s_n}
	\]
and it is known by Theorem \ref{conesGor} that this cone is Gorenstein if and only if there exists a $\c\in\Z^n$ satisfying $(\ref{eq:Crecursion})$. Furthermore, the map 
\[
x \mapsto- \begin{pmatrix}
0 & 0& \dots & 0 & 1 \\
0 & 0& \dots & 1 & 0 \\
\vdots & \vdots & \vdots & \vdots & \vdots \\
1 & 0 &\dots & 0 & 0 \\
\end{pmatrix}x + \overleftarrow{\s}.
\]
shows that $T_{\s}(\Pbn)$ is unimodularly equivalent to $T_{\boldsymbol 0}(\operatorname{\boldsymbol P}^{(\overleftarrow{\s})}_n)$. Therefore, $T_{\boldsymbol 0}(\operatorname{\boldsymbol P}^{(\overleftarrow{\s})}_n)$ is of the form	
	\[
	0 \leq \frac{{\lambda_1}}{\overleftarrow{s_1}} \leq \dots \leq \frac{{\lambda_n}}{\overleftarrow{s_n}}
	\]
which is Gorenstein if and only if there exists a $\d\in\Z^n$ satisfying $(\ref{eq:Drecursion})$. Now solving the recursions for $c_j$ and $d_j$ gives
\begin{footnotesize}
\begin{equation*}
c_j = \frac{\gcd(s_{j-1},s_j)}{s_{j-1}} + \sum_{k=1}^{j-1} \frac{\gcd(s_{k-1},s_k)}{s_{k-1}s_k} \quad \text{and } \quad d_j =  \frac{\gcd(s_{n-j+2},s_j)}{s_{n-j+1}} + \sum_{k=1}^{j-1} \frac{\gcd(s_{n-k+2},s_{n-k+1})}{s_{n-k+2}s_{n-k+1}},
\end{equation*}
\end{footnotesize}
which proves the first claim.

Let us assume that all terms in $(\ref{eq:GorChar})$ are integers. We define
\begin{footnotesize}
\begin{equation*}
c_j = \frac{\gcd(s_{j-1},s_j)}{s_{j-1}} + \sum_{k=1}^{j-1} \frac{\gcd(s_{k-1},s_k)}{s_{k-1}s_k} \quad \text{and } \quad d_j =  \frac{\gcd(s_{n-j+2},s_j)}{s_{n-j+1}} + \sum_{k=1}^{j-1} \frac{\gcd(s_{n-k+2},s_{n-k+1})}{s_{n-k+2}s_{n-k+1}},
\end{equation*}
\end{footnotesize}
where $s_0 = s_{n+1}=1$. In particular, $c_1 = d_1 = 1$.
Since all terms in $(\ref{eq:GorChar})$ are assumed to be integers, both $\c = (c_1,\dots,c_n)$ and $\d = (d_1,\dots,d_n)$ are integer points. By definition, $\c$ satisfies recursion $(\ref{eq:Crecursion})$, and  $\d$ satisfies recursion $(\ref{eq:Drecursion})$.

To show sufficiency, we employ Lemmata~\ref{GorLemma} and ~\ref{lem:H-DescriptionConeOverPolytope}. 
Since the characterization given in Lemma~\ref{GorLemma} essentially requires finding integer solutions to linear equations, we first deduce some divisibility conditions that will later prove useful.
Note that this gives us the following
	\[
	c_n s_{n-1} = c_{n-1}s_n + \gcd(s_{n-1},s_n),
	\]
and 
	\[
	 d_2 \overleftarrow{s}_1 = d_1 \overleftarrow{s}_2 +\gcd(\overleftarrow{s}_{1},\overleftarrow{s}_n) 
	\]
is equivalent to	
	 \[
	  d_2 s_{n} = d_{1}s_{n-1} + \gcd(s_{n-1},s_n) 
	\]
where $d_1 = 1$. Subtracting both equalities, we get 
	\[
	(d_2 + c_{n-1})s_n = (1 + c_n) s_{n-1}.
	\]
Repeating the above process, we also have
	\[
	(d_3+c_{n-2})s_{n-1}=(d_2+c_{n-1})s_{n-2}
	\]
and in general for some $k$, we have
\begin{equation}
\label{eq:dcCondition}
(d_{k+1}+c_{n-k})s_{n-k+1}=(d_k+c_{n-k+1})s_{n-k}.
\end{equation}
If we know that $i=n-k$, then $\gcd(s_{n-k},s_{n-k+1})=1$ and we can deduce that $s_{n-k+1} | (d_k+c_{n-k+1})$ which will be necessary at a later stage of the proof..

By Lemma \ref{GorLemma}, a cone of the form $A \lambda \geq 0$ is Gorenstein if and only if there is a point $\tilde \c \in \cone(\Pbn) \cap \Z^{n+1}$ such that $\alpha^{i}(\tilde\c) = q_i$ for all $i$, where $\alpha^i$ is the $i^{\text{th}}$ row of $A$ and $q_i$ is defined as in Lemma \ref{GorLemma}. Lemma \ref{lem:H-DescriptionConeOverPolytope} explicitly describes the rows, implying
\[
q_1 = \frac{1}{s_1},  q_2 = \frac{1}{\operatorname{lcm}(s_1,s_2)}, \dots, q_n = \frac{1}{{\operatorname{lcm}(s_{n-1},s_{n-2})}}, q_{n+1} = \frac{1}{s_n}.
\]
We will show that $\tilde \c = (\c,h)$, where $\c$ is defined as above and where $h \in \Z_{\geq 1}$, satisfies \rem{ In particular, $\tilde c_i = c_i$ for $1 \leq i \leq n$ and $\tilde c_{n+1}=h$. By Lemma~\ref{GorLemma}, we need to verify that the point $\tilde \c\in \Z^{n+1}$ satisfies} $\alpha^{i}(\tilde\c) = q_i$ for all $i =1, \dots, n+1$. The conditions $\alpha^{i} (\tilde \c) = q_i$ for $i =1, \dots, n$ are equivalent to saying $c_1 = 1$ and that
\[
c_j s_{j-1} = c_{j-1}s_j + \gcd(s_{j-1},s_j)
\]
for $2\leq j \leq n$. These conditions are all satisfied by assumption. However, we also need to satisfy the condition 
\[
\frac{-c_n}{s_n} +\tilde c_{n+1} = \frac{1}{s_n}, 
\]
or equivalently
\[
 hs_n = 1 + c_n.
\]
Now, we note that from Equation (\ref{eq:dcCondition}) it follows that
	\[
	s_{n}=\frac{(1+c_n)}{(d_2+c_{n-1})}s_{n-1},
	\]
so we can rewrite 
	\[
	hs_{n-1} = d_2+c_{n-1}.
	\]	
We can iterate these substitutions repeatedly to arrive at the equality 
	\[
	hs_{n-k+1}=d_k+c_{n-k+1}.
	\]
However, since $s_{n-k+1} | (d_k+c_{n-k+1})$, $h$ is an integer. Here we are implicitly using that $\c, \d \in \Z_{\geq 1}^n$, which follows from the recursive definition. So we are done. 
\end{proof}
\rem{
\flocomment{Remove the following remark}
\begin{remark}
Let $\s\in \Z_{\geq 1}^n$ be as above. Due to the recursive nature of $\c$ ($\d$, respectively), one can successively solve for $c_j$ ($d_{j}$, respectively) to obtain the following characterization: The polytope $\Pbn$ is Gorenstein if and only if for all $j \geq 2$
\begin{equation*}
c_j = \frac{\gcd(s_{j-1},s_j)}{s_{j-1}} + \sum_{k=1}^{j-1} \frac{\gcd(s_{k-1},s_k)}{s_{k-1}s_k} \quad \text{and } d_j=\quad \frac{\gcd(s_{n-j+2},s_j)}{s_{n-j+1}} + \sum_{k=1}^{j-1} \frac{\gcd(s_{n-k+2},s_{n-k+1})}{s_{n-k+2}s_{n-k+1}}
\end{equation*}
are integers where $s_0 = s_{n+1}=1$.
\end{remark}}
\rem{
\begin{remark}
\label{rem:VertexConeGorenstein}
The first direction of the proof holds more generally. It shows that if $P$ is Gorenstein, then $T_{\v}(P)$ is Gorenstein for any vertex. In particular, the forward direction of Theorem~\ref{GorensteinChar} holds independently of any assumptions on $\s$. Theorem~\ref{GorensteinChar} says that if the two vertex cones at $\v_1 =\mathbf 0$ and at $\v_2 = \s$ are Gorenstein, then the converse holds, i.e., $\Pbn$ is already Gorenstein. In this case $\c$ ($\d$ respectively) is the Gorenstein point of $T_{\mathbf 0}$ ($T_{\s}$ respectively).
\end{remark}}

\begin{remark}
We mentioned before that Theorem~\ref{GorensteinChar} applies to a large subfamily of $\s$-lecture hall polytopes. This remark will make this statement more precise. Given two positive integers $a$ and $b$, the probability that $\gcd(a,b)=1$ converges to $\frac{1}{\zeta(2)} = \frac{6}{\pi^2}$, where $\zeta(s) = \sum_{n=1}^{\infty} \frac{1}{n^s}$ is the \emph{Riemann $\zeta$-function}, see \cite[Theorem 332]{HardyWright}. Heuristically, assuming that these events are independent (which they are not), we then get that roughly $\left(1-\frac{6}{\pi^2}\right)^{n-1}$-percent of sequences fall within the range of our theorem. Computer simulations running $10,000,000$ repetitions  per dimension $n$ with parameters  $1\leq a,b \leq 10,000,000$ and $15 \leq n \leq 50$ suggest that this estimate is fairly precise.
\end{remark}

\begin{remark}
Let $\s\in \Z^n$ be such that $T_{\mathbf 0}(\Pbn)$ is Gorenstein with Gorenstein point $\c$. In \cite[Corollary 2.7]{BeckEtAl-GorensteinLHC}, the authors remark that the truncated sequence $(s_1,s_2,\dots,s_i)$ gives rise to a Gorenstein cone \begin{tiny}$T_{\mathbf 0}(\mathbf{P}^{((s_1,s_2,\dots,s_i))}_i)$\end{tiny} with Gorenstein point \begin{scriptsize}$(c_1,c_2,\dots,c_i)$\end{scriptsize}. However, the direct analogue of this statement is not true in our case. The sequence \begin{tiny}$(8,6,10,10,5,2,4)$\end{tiny} gives rise to a Gorenstein $\s$-lecture hall polytope, whereas \begin{tiny}$(8,6,10,10,5)$\end{tiny} does \emph{not} give rise to a Gorenstein $\s$-lecture hall polytope, since it has 39 interior lattice points.  
\end{remark}

Theorem~\ref{GorensteinChar} along with Theorem \ref{conesGorPairwise} implies the following more specialized characterization.

\begin{corollary}
Let $\s = (s_1,s_2,\dots,s_n)\in \Z_{\geq 1}^n$ such that $\gcd(s_i,s_{i+1})=1$ for all $1\leq i<n$. Then $\Pbn$ is Gorenstein if and only if $\s$ and $\overleftarrow{\s}$ are $\u$-generated sequences.
\end{corollary}

We have the following corollary on the level of $\s$-Eulerian polynomials

\begin{corollary}
Let $\s = (s_1,s_2,\dots,s_n)\in \Z_{\geq 1}^n$ such that $\gcd(s_i,s_{i+1})=1$ for some $1\leq  i< n$. The $\s$-Eulerian polynomial is palindromic if and only if there exist $\c,\d\in\Z^n$ satisfying
	\[
	c_js_{j-1}=c_{j-1}s_j+\gcd(s_{j-1},s_{j})
	\]
and 
	\[
	d_j\overleftarrow{s_{j-1}}=d_{j-1}\overleftarrow{s_j}+\gcd(\overleftarrow{s_{j-1}},\overleftarrow{s_{j}})
	\]	
for $j>1$ with $c_1=d_1=1$.	
\end{corollary}
Table \ref{table:s-Eulerian} contains some examples of palindromic $\s$-Eulerian polynomials.
We list them with the corresponding $\c$ and $\d$ sequences to more immediately see why $\s$ satisfies the integrality condition specified by Theorem~\ref{GorensteinChar}.

\begin{table}[h!]
\centering
\footnotesize
	\begin{tabular}{|c|c|c|c|c|}
	\hline
	 & sequence $\s$ & corresponding $\c$ & corresponding $\d$ & $\s$-Eulerian polynomial\\
	 \hline\hline
	  & & & &  \\
	 (i) & $(2,1,3,2,1)$ & $(1,1,4,3,2)$ & $(1,3,5,2,5)$ & $1+5z+5z^2+z^3$\\
	 & & &  & \\
	 \hline
	  & & & & \\
	 (ii) & $(3,2,3,1,2)$ & $(1,1,2,1,3)$ & $(1,1,4,3,5)$ & $1+9z+16z^2+9z^3+z^4$\\ 
	 & & & & \\
	 \hline
	 & & & & \\
	 (iii) & $(1,4,3,2,3)$ & $(1,5,4,3,5)$ & $(1,1,2,3,1)$ & $1+ 16z+38z^2+16z^3+z^4$\\
	 & & & & \\
	 \hline
	 & & & & \\
	 (iv) & $(3,5,2,3,1)$ & $(1,2,1,2,1)$ & $(1,4,3,8,5)$ & $1+20z+48z^2+20z^3+z^4$\\
	 & & & & \\
	 \hline
	 & & & & \\
	 (v) & $(1,2,3,4,5)$ & $(1,3,5,7,9)$ & $(1,1,1,1,1)$ & $1+26z+66z^2+26z^3+z^4$\\
	 & & & &\\
	 \hline
	 & & & & \\
	 (vi) & $(1,2,5,8,3)$ & $(1,3,8,13,5)$ & $(1,3,2,1,1)$ & $1+50z+138z^2+50z^3+z^4$\\
	 & & & &  \\
	 \hline
	 & & & & \\
	 (vii) & $(4,3,2,5,3)$ & $(1,1,1,3,2)$ & $(1,2,1,2,3)$ &  $1+30z+149z^2+149z^3+30z^4+z^5$\\
	 & & & &\\
	 \hline
	 & & & & \\
	 (viii) & $(4,7,3,2,3)$ & $(1,2,1,1,2)$ & $(1,1,2,5,3)$ &   \begin{scriptsize}$1+43z+208z^2+208z^3+43z^4+z^5$\end{scriptsize}\\
	 & & & & \\
	 \hline
	 & & & &  \\ 
	 (ix) & $(5,9,4,3,2)$ & $(1,2,1,1,1)$ & $(1,2,3,7,6)$ &  \begin{scriptsize}$1+82z+457z^2+457z^3+82z^4+z^5$\end{scriptsize}\\
	 & & & & \\
	 \hline
	 & & & & \\
	 (x) & $(3,5,12,7,2)$ & $(1,2,5,3,1)$ & $(1,4,7,3,2)$ & \begin{scriptsize} $1+175z+1084z^2+1084z^3+175z^4+z^5$\end{scriptsize}\\
	 & & & & \\ 
	 \hline
	 & & & & \\
	 (xi) & $(3,11,8,5,2)$ & $(1,4,3,2,1)$ & $(1,3,5,7,2)$ &  \begin{scriptsize}$1+180z+1139z^2+1139z^3+180z^4+z^5$\end{scriptsize}\\
	 & & & & \\
	 \hline 
	 & & & &\\
	 (xii) & $(2,7,5,10,4)$ & $(1,4,3,7,3)$ & $(1,3,2,3,1)$ &  \begin{scriptsize}$1+181z+1218z^2+1218z^3+181z^4+z^5$ \end{scriptsize}\\
	 & & & &\\
	 \hline
	 & & & & \\
	 (xiii) & $(3,8,13,5,2)$ & $(1,3,5,2,1)$ & $(1,3,8,5,2)$ & \begin{scriptsize}$1+213z+1346z^2+1346z^3+213z^4+z^5$\end{scriptsize}\\
	 & & & & \\
	 \hline
	\end{tabular}
	\medskip
\normalsize
\caption{Palindromic $\s$-Eulerian Polynomials.}
\label{table:s-Eulerian}
\end{table}

\commentout{
\begin{table}[h!]
\centering
\footnotesize
	\begin{tabular}{|c|c|c|}
	\hline
	 & sequence $\s$ & $\s$-Eulerian polynomial\\
	 \hline\hline
	  & &  \\
	 (i) & $(2,1,3,2,1)$ & $1+5z+5z^2+z^3$\\
	 & & \\
	 \hline
	  & & \\
	 (ii) & $(3,2,3,1,2)$ & $1+9z+16z^2+9z^3+z^4$\\ 
	 & & \\
	 \hline
	 & & \\
	 (iii) & $(1,4,3,2,3)$ & $1+ 16z+38z^2+16z^3+z^4$\\
	 & & \\
	 \hline
	 & & \\
	 (iv) & $(3,5,2,3,1)$ & $1+20z+48z^2+20z^3+z^4$\\
	 & & \\
	 \hline
	 & & \\
	 (v) & $(1,2,3,4,5)$ & $1+26z+66z^2+26z^3+z^4$\\
	 & &\\
	 \hline
	 & & \\
	 (vi) & $(1,2,5,8,3)$ & $1+50z+138z^2+50z^3+z^4$\\
	 & & \\
	 \hline
	 & & \\
	 (vii) & $(4,3,2,5,3)$ & $1+30z+149z^2+149z^3+30z^4+z^5$\\
	 & &\\
	 \hline
	 & & \\
	 (viii) & $(4,7,3,2,3)$ & $1+43z+208z^2+208z^3+43z^4+z^5$\\
	 & & \\
	 \hline
	 & & \\ 
	 (ix) & $(5,9,4,3,2)$ & $1+82z+457z^2+457z^3+82z^4+z^5$\\
	 & & \\
	 \hline
	 & & \\
	 (x) & $(3,5,12,7,2)$ & $1+175z+1084z^2+1084z^3+175z^4+z^5$\\
	 & & \\ 
	 \hline
	 & & \\
	 (xi) & $(3,11,8,5,2)$ & $1+180z+1139z^2+1139z^3+180z^4+z^5$\\
	 & & \\
	 \hline 
	 & & \\
	 (xii) & $(3,8,13,5,2)$ & $1+213z+1346z^2+1346z^3+213z^4+z^5$\\
	 & & \\
	 \hline
	\end{tabular}
	\medskip
\normalsize
\caption{Palindromic $\s$-Eulerian Polynomials.}
\label{table:s-Eulerian}
\end{table}
}

\section{Characterization of level $\s$-lecture hall polytopes}\label{sec:Level}
We now give a characterization of $\s$-sequences that admit level $\Pbn$, which is given in terms of the structure of $\s$-inversion sequences.
We recall two definitions from Section~\ref{sec:Intro}. Let $\Ib_{n,k}^{(\s)}\coloneqq\{\e\in\Ibn \, : \asc(\e)=k\}$ be the set of inversion sequences with exactly $k$ ascents. Furthermore, for inversion sequences $\e=(e_1,e_2,\ldots, e_n), \e'=(e_1',e_2',\ldots, e_n')\in\Ibn$, we define $\e+\e'=(e_1+e_1' \bmod s_1,e_2+e_2'\bmod s_2,\ldots, e_n+e_n'\bmod s_n)$.

\commentout{
}

\subsection{Proof of Theorem \ref{THM:LEVEL}}\label{sec:Level:Proof}
Our proof relies on understanding the link between the arithmetic structure of inversion sequences and the semigroup structure of lattice points in $\Pi_{\Pbn}$.
To fully understand and exploit this connection, we will need several lemmata. 
For notation, let $V(\Pbn)=\{\v_0,\ldots,\v_{n}\}$ denote the set of vertices of $\Pbn$ and let $\mathscr{P}_n^{(\s)}\coloneqq(\Pbn\cap \Z^n)-V(\Pbn)$.

\begin{lemma}\label{height1}
 There is an explicit bijection
	\[
	\phi:\mathscr{P}_n^{(\s)} \longrightarrow \Ib_{n,1}^{(\s)}
	\]
where  $\phi(\lambda_1,\ldots,\lambda_n)= (e_1,\ldots, e_n)$ given by $e_i=s_i-\lambda_i (\bmod s_i)$. 	
\end{lemma}

\begin{proof}
Let $\boldsymbol\lambda=(\lambda_1,\lambda_2,\ldots,\lambda_n)\in \mathscr{P}_n^{(\s)}$. We have that
	\[
	0\leq\frac{\lambda_1}{s_1}\leq\frac{\lambda_2}{s_2}\leq\ldots\leq \frac{\lambda_n}{s_n}\leq 1.
	\]
Note that this means that $\lambda_i\leq s_i$ for all $i$ and if $\lambda_i=s_i$ then $\lambda_j=s_j$ for all $i\leq j\leq n$.
Additionally, note that the vertices of $\Pbn$ are precisely the lattice points of the form
	\[
	(0,\ldots,0,s_i, s_{i+1},\cdots, s_n).
	\]	
So, then $\boldsymbol\lambda$ can be expressed as:
	\[
	\boldsymbol\lambda=(0,\ldots,0, a_i, a_{i+1},\ldots, a_j,s_{j+1},\ldots,s_n),
	\]	
where each $0<a_k<s_k$.

If we apply our map $\lambda_i\mapsto s_i-\lambda_i   \,\,(\operatorname{mod} s_i)$, we get the inversion sequence
	\[
	\e=(0,0,\ldots,0, s_i-a_i,s_{i+1}-a_{i+1},\ldots, s_j-a_j,0, \ldots, 0).
	\]	
It is left to verify that $\e\in\Ib_{n,1}^{(\s)}$. Since we know that 
	\[
	0<\frac{a_i}{s_i}\leq\frac{a_{i+1}}{s_{i+1}}\leq \ldots\leq \frac{a_j}{s_j}<1
	\]	
which holds if and only if 
	\[
	1>\frac{s_i-a_i}{s_i}\geq \frac{s_{i+1}-a_{i+1}}{s_{i+1}}\geq \ldots \geq \frac{s_j-a_j}{s_j}>0,
	\]	
we know that $\e$ contains exactly one ascent at position $i-1$.

This process is certainly reversible, so we have a bijection.	
\end{proof}

Note that $\mathscr{P}_n^{(\s)}$ contains precisely the elements at height 1 in $\Pi_{\Pbn}$. Next, we will extend $\varphi^{-1}$ to establish a bijection, $\psi$, between $\Pi_{\Pbn}$ and $ \Ib_{n}^{(\s)}$. 

\begin{definition}
\label{def:Bijection}
Let $\psi \colon \Ib_{n}^{(\s)} \to \Pi_{\Pbn} $ and let $\e=(e_1,e_2,\cdots,e_n)\in\Ib_{n,k}^{(\s)}$, where the $k$ ascents are at positions $i_1,i_2,\cdots,i_k$. Moreover, we set $i_{k+1} \coloneqq n$. Then, for $\ell \in [k]$ and $i_{\ell}<j \leq i_{\ell+1}$,  we define
	\[
	\psi(\e)_j= \ell\cdot s_j-e_j.
	\]
\end{definition}

\begin{lemma}\label{correspondence}
The map $\psi \colon \Ib_{n}^{(\s)} \to \Pi_{\Pbn} \cap \Z^{n+1}$ is a bijection and $\psi ( \Ib_{n,k}^{(\s)}) = \{\x \in \Pi_{\Pbn} \cap \Z^{n+1} \colon \x_{n+1}=k \}$.
Moreover, if $\f \in \Ib_{n,k-1}^{(\s)}$ and $\g \in \Ib_{n,1}^{(\s)}$, then $\f+\g  \in \Ib_{n,k}^{(\s)}$ if and only if $\psi(\f)+\psi(\g)\in \Pi_{\Pbn}$.
 \commentout{vector-addition of the elements in $\Pi_{\Pbn}\cap \Z^{n+1}$ corresponds to entry-wise addition of the inversion sequences modulo $s_i$ in the $i$th position. That is, any decomposition of $\lambda$ as a sum of elements of height one in $\Pi_{\Pbn}$ is consistent with the sum of inversion sequences.}
\end{lemma}
\begin{remark}
The map $\psi$ is related to the bijective map $\overline{\operatorname{REM}}$ in \cite{LiuStanley}. To be precise, for $1 \leq j \leq n$, we have $\psi(\e)_j = \overline{\operatorname{REM}}^{-1}(\e)_j$  . There are two novel parts of Lemma~\ref{correspondence}. First, it establishes that elements with $k$ \emph{ascents} get mapped to height $k$, which generalizes \cite[Corollary 6.2]{LiuStanley}. It is curious to compare this to \cite[Corollary 3.8]{LiuStanley}, where Liu--Stanley define a map $\operatorname{REM}$ that bijectively maps integer points of  $\Pi_{\Pbn} $ of height $k$ to elements of $\Ib_{n}^{(\s)} $ with $k$ \emph{descents}. The second novel result is that  $\f+\g  \in \Ib_{n,k}^{(\s)}$ if and only if $\psi(\f)+\psi(\g)\in \Pi_{\Pbn}$. This second part will be crucial in the proof of Theorem~\ref{THM:LEVEL}.
\end{remark}
\commentout{
\begin{remark}
By entry-wise addition of the inversion sequences modulo $s_i$ in the $i$th position, we mean that we pick the unique representative of this equivalence class in $\{0,1,\dots,s_i -1 \}$.
\end{remark}}

\begin{proof}[Proof of Lemma \ref{correspondence}]
It is clear that map from Definition~\ref{def:Bijection} is injective. We must verify the following:
	\begin{enumerate}
	\item[(A)] The image of $\boldsymbol\lambda$ under the map from Definition~\ref{def:Bijection} is an element of $\Pi_{\Pbn}$.
	\item[(B)] Entry-wise addition of inversion sequences is consistent with addition in the semigroup $\cone(\Pbn) \cap \Z^{n+1}$.
	\end{enumerate}

To show (A), note that it is clear that $\boldsymbol\lambda$ is at height $k$ in $\R^{n+1}$. Moreover, it is even clear that $(\lambda_1,\cdots,\lambda_n)\in k\cdot \Pbn\cap \Z^n$, as $i_\ell<t<i_{\ell+1}$ and $\displaystyle \frac{e_t}{s_t}\geq \frac{e_{t+1}}{s_{t+1}}$ imply that  
	\[
	\frac{\ell\cdot s_t-e_t}{k\cdot s_t}\leq \frac{\ell\cdot s_{t+1}-e_{t+1}}{k\cdot s_{t+1}},
	\]	
and if $t=i_{\ell+1}$, 
	\[
	\frac{\ell\cdot s_t-e_t}{k\cdot s_t}\leq \frac{(\ell+1)\cdot s_{t+1}-e_{t+1}}{k\cdot s_{t+1}}
	\]	
is immediate from $e_{t+1}<s_{t+1}$.

To verify that $\boldsymbol\lambda$ is in fact in $\Pi_{\Pbn}$, we must show that neither of the following hold:
	\begin{itemize}
	\item[(i)] $(\lambda_1,\cdots,\lambda_n)\in (k-1)\cdot \Pbn\cap \Z^n$
	\item[(ii)] $\boldsymbol\lambda=\boldsymbol\lambda'+\v$ where $(\lambda_1',\cdots,\lambda_n')\in(k-1)\cdot \Pbn\cap \Z^n$ and $\v$ is a vertex of $\Pbn$.
	\end{itemize}

Note that (i) is impossible as we have $\lambda_n=k\cdot s_n-e_n>(k-1)s_n$ because $e_n<s_n$. For (ii), suppose that we write $\boldsymbol\lambda=\boldsymbol\lambda'+\v$, where $\v=(0,0,\cdots,0,s_{j+1},\cdots, s_n)$ with $0\leq j<n$. There are two possible cases: $j\in \Asc(\e)$ or $j\not\in\Asc(\e)$. If $j\in \operatorname{Asc}(\e)$, then $\displaystyle \frac{e_j}{s_j}<\frac{e_{j+1}}{s_{j+1}}$. Consider $\boldsymbol\lambda'$ and suppose that  
	\[
	(\lambda_1',\cdots,\lambda_n')=(\lambda_1,\cdots,\lambda_j,\lambda_{j+1}-s_{j+1},\cdots, \lambda_n-s_n)\in(k-1)\cdot\Pbn\cap\Z^n.
	\]	
Given that $\lambda_j=(p-1)\cdot s_j -e_j$ and $\lambda_{j+1}=p\cdot s_{j+1}-e_{j+1}$ where $j$ is the $p$th ascent $i_p$, 
	\[
	\frac{(p-1)\cdot s_j-e_j}{(k-1)s_j}\leq \frac{p\cdot s_{j+1}-e_{j+1}-s_{j+1}}{(k-1)s_{j+1}}=\frac{(p-1)\cdot s_{j+1}-e_{j+1}}{(k-1)s_{j+1}}.
	\]	
However, this is equivalent to $\displaystyle\frac{e_j}{s_j}\geq \frac{e_{j+1}}{s_{j+1}}$ so this cannot occur. 	If $j\not\in\Asc(\e)$, say that $j>i_p$, the location of the $p$th ascent $i_p$, so $\lambda_j=p\cdot s_j-e_j$ and $\lambda_{j+1}=p\cdot s_{j+1}-e_{j+1}$. For $(\lambda_1',\cdots,\lambda_n')\in(k-1)\cdot\Pbn\cap \Z^n$, 
	\[
	\frac{p\cdot s_j-e_j}{(k-1)s_j}\leq \frac{p\cdot s_{j+1}-e_{j+1}-s_{j+1}}{(k-1)s_{j+1}}=\frac{(p-1)\cdot s_{j+1}-e_{j+1}}{(k-1)s_{j+1}}.
	\]
This inequality is equivalent to $\displaystyle \frac{e_j}{s_j}\geq \frac{e_{j+1}}{s_{j+1}}+1	$ which is a contradiction to $e_j<s_j$.

Therefore, we have shown (A). Note that this is sufficient for showing the bijection, as the map is clearly injective and the sets are of the same cardinality by previous work of Savage and Schuster \cite{SavageSchuster}. That said, the bijection can also be realized through the fact that the map from Lemma~\ref{correspondence} can clearly be reversed. In particular, suppose that $\boldsymbol\lambda=(\lambda_1,\cdots,\lambda_n,k)\in \Pi_{\Pbn}$, we get our inversion sequence $\e$ by
	\[
	e_i=-\lambda_i \mod s_i.
	\]  
Note that this inversion sequence will have precisely $k$ ascents and moreover the $p$th ascent in the sequence will occur at $i$ precisely when $(p-1)\cdot s_i\leq \lambda_i <p\cdot s_i$ and $p\cdot s_{i+1}\leq \lambda_{i+1}< (p+1)\cdot s_{i+1}$ for some $1\leq p\leq k-1$. This is the exact reversal of the constructive map from inversion sequences to lattice points is $\Pi_{\Pbn}$.

To show (B), suppose that $\f\in\Ib_{n,k-1}^{(\s)}$ and $\g\in \Ib_{n,1}^{(\s)}$ such that $\f+\g=\e\in\Ib_{n,k}^{(\s)}$. So
	\[
	\f=(f_1,\ldots,f_{j-1},f_j,\ldots, f_h, f_{h+1},\ldots, f_n)
	\]
	and
	\[
	\g=(0,\ldots,0,g_j,\ldots,g_h,0,\ldots,0)
	\]
	and 
	\[
	\e=(f_1,\ldots,f_{j-1},(f_j+g_j)\bmod s_j,\ldots,(f_h+g_h)\bmod s_h, f_{h+1},\ldots,f_n).
	\]

Consider the corresponding lattice points for $\f$ and $\g$ in $\Pi_{\Pbn}$:
	\[
	\boldsymbol\lambda_\f=(\lambda_{\f_1},\ldots, \lambda_{\f_{j-1}},\lambda_{\f_j},\ldots,\lambda_{\f_h},\lambda_{\f_{h+1}},\ldots,\lambda_{\f_n},k-1)
	\]
	and 
	\[
	\boldsymbol\lambda_\g=(0,\ldots,0,\lambda_{\g_j},\ldots,\lambda_{\g_h},s_{h+1},\ldots, s_n, 1).
	\]
Adding these lattice points in the semigroup yields
	\[
	\boldsymbol\lambda_\f+\boldsymbol\lambda_\g=(\lambda_{\f_1},\ldots, \lambda_{\f_{j-1}},\lambda_{\f_j}+\lambda_{\g_j},\ldots,\lambda_{\f_h}+\lambda_{\g_h},\lambda_{\f_{h+1}}+s_{h+1},\ldots,\lambda_{\f_n}+s_n,k).
	\]	
We have two possible cases: either  $\boldsymbol\lambda_\f+\boldsymbol\lambda_\g\in\Pi_{\Pbn}$ or  $\boldsymbol\lambda_\f+\boldsymbol\lambda_\g\not\in\Pi_{\Pbn}$.

If $\boldsymbol\lambda_\f+\boldsymbol\lambda_\g\in\Pi_{\Pbn}$, we consider the reverse map which will give the inversion sequence
\begin{footnotesize}	\[
	(\ldots, \lambda_{\f_{j-1}}\bmod s_{j-1},-(\lambda_{\f_j}+\lambda_{\g_j})\bmod s_j,\ldots,-(\lambda_{\f_h}+\lambda_{\g_h})\bmod s_h,-(\lambda_{\f_{h+1}}+s_{h+1})\bmod s_{h+1},\ldots)
	\]
\end{footnotesize}
and this inversion sequence is precisely $\e=\f+\g$, as desired.

Now suppose that $\boldsymbol\lambda_\f+\boldsymbol\lambda_\g\not\in\Pi_{\Pbn}$.
 Note that we can express $\boldsymbol\lambda_\f+\boldsymbol\lambda_\g=\lambda'+\sum_{i=1}^n \alpha_i v_i$ where $\boldsymbol\lambda'\in\Pi_{\Pbn}$, there is at least one  $\alpha_i \in \Z_{\geq 1}$, and  $\boldsymbol\lambda'$ is at height $r<k$. 
Additionally, given that $\v_i=(0,\ldots,0,s_i,s_{i+1},\ldots, s_n)$, it is clear that $\boldsymbol\lambda_\f+\boldsymbol\lambda_\g$ maps to the same inversion sequence as $\boldsymbol\lambda'$ by definition of the inverse map. 
This implies that $\e$ maps to $\lambda'$ and thus $\e\in\Ib_{n,r}^{(\s)}$ for $r<k$, which contradicts our initial assumption.
\end{proof}

\begin{remark}
We should note that in the proof about the compatibility of addition, we consider only inversion sequences $\f\in \Ib_{n,k-1}^{(\s)}$ and $\g\in \Ib_{n,1}^{(\s)}$ such that $\f+\g\in\Ib_{n,k}^{(\s)}$, as this is the requirement for staying inside the fundamental parallelepiped. However, this need not always be the case. If $\f+\g\in\Ib_{n,\ell}^{(\s)}$ for some $\ell\leq k-1$, the addition of the sequences is still consistent with addition in the semigroup, but this occurrence is precisely when $\boldsymbol\lambda_\f+\boldsymbol\lambda_\g\not\in\Pi_{\Pbn}$. In particular, $\boldsymbol\lambda_\f+\boldsymbol\lambda_\g=\boldsymbol\lambda_{\f+\g}+(0,\cdots,0,k-\ell)$, which lies in the equivalence class  $\boldsymbol\lambda_{\f+\g}$, but is not the representative in $\Pi_{\Pbn}$.
\end{remark}

With this understanding of the arithmetic structure of $\Ibn$, we can now give a proof of the  characterization.

\begin{proof}[Proof of Theorem \ref{THM:LEVEL}]
Consider the semigroup algebra $k[\Pbn]\coloneqq k[\cone(\Pbn)\cap \Z^{n+1}]$. We recall that \rem{if ${k[\Pbn]}$ is a graded, \local, Cohen-Macaulay algebra with $\dim({k[P]})=d$, then }$k[\Pbn]$ is level if for some homogeneous system of parameters $\theta_1,\ldots,\theta_d$ of ${k[\Pbn]}$, all the elements of the graded vector space $\soc({k[\Pbn]}/(\theta_1,\ldots,\theta_d))$ are of the same degree.
Notice  that  $\Pbn$ is a simplex and let $\Pi_{\Pbn}$ denote the (half-open) fundamental parallelepiped.
Note that $\dim(k[\Pbn])=n+1$ and $k[\Pbn]$ has a natural homogeneous system of parameters, namely the monomials corresponding to the vertices, which we denote by $\theta_0,\theta_1,\ldots, \theta_n$.
 The quotient $k[\Pbn]/(\theta_0,\cdots,\theta_{n})$ contains precisely the equivalence classes of lattice points in $\Pi_{\Pbn}$. Let $m_1,\cdots,m_{\alpha} \in\Pi_{\Pbn}$ be the elements at height 1.
The socle $\operatorname{soc}(k[\Pbn]/(\theta_0,\cdots,\theta_{n}))$ are precisely the lattice points in $\lambda\in \Pi_{\Pbn}$ such that $\lambda+m_i\not\in\Pi_{\Pbn}$ for all $m_i$ by Lemma \ref{height1} and Theorem \ref{IDP}. By Lemma \ref{correspondence}, we know that semigroup addition corresponds to entry-wise addition on inversion sequences. 
Subsequently, this condition on inversion sequences is precisely the condition that only elements of highest degree in $\Pi_{\Pbn}$ are in $\soc((k[\Pbn]/(\theta_0,\cdots,\theta_{n}))$, which then must contain elements that are all the same degree.
  
\end{proof}

\subsection{Consequences of Theorem \ref{THM:LEVEL}}\label{sec:Level:Consequences}
First consider the following resulting inequalities given for the coefficients of $\s$-Eulerian polynomials.
\begin{corollary}
Let $\s=(s_1,s_2,\ldots,s_n)$ be a sequence such that $\Pbn$ is level. Then the coefficients of the $\s$-Eulerian polynomial $h^\ast(\Pbn,z)=1+h_1^\ast z+\cdots+h_r^\ast z^r$ satisfy  the inequalities $h_i^\ast\leq h_j^\ast h_{i+j}^\ast$ for all pairs $i$ and $j$ such that $h_{i+j}^\ast>0$.
\end{corollary}

These inequalities follow from \cite[Chapter III. Proposition 3.3]{StanleyGreenBook} and provide additional information of the behavior of $\s$-Eulerian polynomials to complement the known log-concave inequalities from \cite{SavageVisontai}. It is worth noting that these inequalities need not be satisfied for arbitrary $\s$.  For example, the sequence $\s=(2,3,5,9)$ does not give rise to a level polytope as there exists no element $\f\in\Ib_{4,1}^{(2,3,5,9)}$ such that $\f+\e\in \Ib_{4,4}^{(2,3,5,9)}$ for the inversion sequence $\e=(1,1,2,4)\in\Ib_{4,3}^{(2,3,5,9)}$. Moreover, 
	\[
	h^\ast(\Pb_4^{(2,3,5,9)},z)=1+48z+154z^2+66z^3+z^4
	\] 
and we notice that $h_3^\ast >h_1^\ast h_4^\ast$.

In addition to the Gorenstein characterization given in Section \ref{sec:Gorenstein}, we also arrive at a different characterization by considering the following restriction of Theorem \ref{THM:LEVEL}.

\begin{corollary}\label{secondaryGorenstein}
Let $\s\in\Z_{\geq 1}^n$ and let $r=\max\{\asc(\e) \, : \, \e\in \Ibn \}$. Then $\Pbn$ is Gorenstein if and only if for any $\e\in\Ib_{n,k}^{(\s)}$ with $1\leq k <r$ there exists some $\e'\in \Ib_{n,1}^{(\s)}$ such that $(\e+\e')\in \Ib_{n,k+1}^{(\s)}$ and $|\Ib_{n,r}|=1$.
\end{corollary}
\begin{proof}
$\Pbn$ is Gorenstein if and only if $\Pbn$ is level with exactly one canonical module generator. The canonical module of $k[\Pbn]$ for $\Pbn$ level has $|\Ib_{n,r}^{(\s)}|$ generators, as this is the leading coefficient of the $h^\ast$ polynomial of $\Pbn$.  
\end{proof}

We should note that in general Corollary~\ref{secondaryGorenstein} is less computationally useful than Theorem~\ref{GorensteinChar}. 
However, it is unexpected that conditions from Corollary~\ref{secondaryGorenstein}  and Theorem~\ref{GorensteinChar} are equivalent when there exists an index $i$ such that $\gcd(s_{i-1},s_{i})=1$. 
Moreover, Corollary \ref{secondaryGorenstein} has the added benefit of providing a characterization with no explicit restrictions on $\s$.

In the case of $\s\in\Zpos^2$, the conditions of Theorem \ref{THM:LEVEL} must always be satisfied. Therefore, we have the following result. 

\begin{corollary}\label{twodim}
The $\s$-lecture hall polytope $\Pb_2^{(s_1,s_2)}$ is level for any $\s=(s_1,s_2)$.
\end{corollary}
\begin{remark}
By \cite[Proposition 1.2]{HigashitaniYaganawa}, every lattice polygon is level. We state this result only to illustrate that one can explicitly use Theorem \ref{THM:LEVEL} to determine levelness, especially in small dimensions.
\end{remark}

The characterization from Theorem~\ref{THM:LEVEL} allows for the construction of new level $\s$-lecture hall polytopes through the following corollaries.

\begin{corollary}\label{1corolary}
The $\s$-lecture hall polytope $\Pbn$ is level if and only if the $\s$-lecture hall polytope $\Pb_{n+1}^{(1,\s)}$ is level.
\end{corollary}

\begin{proof}
We can express any inversion sequence $\e\in\Ib_{n+1}^{(1,\s)}$ as 
	\[
	\e=(0 , \e') 
	\]
where $\e'\in\Ibn$. Thus, $\e$ satisfies the conditions of Theorem \ref{THM:LEVEL} exactly when $\e'$ satisfies the conditions.	
\end{proof}

\begin{remark}
One also has that $\Pbn$ is level if and only if $\Pb_{n+1}^{(\s,1)}$ is level by applying an analogous argument. Corollary~\ref{1corolary} can also be directly proven, as $\Pb_{n+1}^{(1,\s)}$ is the lattice pyramid over $\Pbn$, which is level if and only if $\Pbn$ is level. 
\end{remark}

\begin{corollary}\label{freeproductresult}
If both $\Pbn$ and $\Pb_m^{(\t)}$ are level, then $\Pb_{n+m+1}^{(\s,1,\t)}$ is level.
\end{corollary}

\begin{proof}
Any inversion sequence $\e\in\Ib_{n+m+1}^{(\s,1,\t)}$ can expressed as 
	\[
	\e=(\e_1 , 0 , \e_2)
	\]
where $\e_1\in\Ibn$ and $\e_2\in\Ib_m^{(\t)}$. Subsequently, $\e$ satisfies the conditions of Theorem \ref{THM:LEVEL} when $\e_1$ and $\e_2$ both satisfy the conditions of Theorem \ref{THM:LEVEL}.  	
\end{proof}

\begin{remark}
$\s$-lecture hall polytopes similar to the ones in Corollary~\ref{1corolary} and \ref{freeproductresult} have been studied, e.g., in \cite{LiuSolus}. In \cite[Theorem 4.3]{LiuSolus}, the authors show that in every dimension $n \geq 3$, there is an $\s$-lecture hall polytope $\Pbn$ with $\s = (1,\dots,1,a,1,\dots,1,b,1\dots,1)$ with $k_1 \geq 0$ many leading, $k_2 \geq 1$ many intermediate, and $k_3 \geq 0$ many trailing $1$'s such that the Ehrhart polynomial $i(\Pbn,t)$ has at least one negative coefficient.
\end{remark}
\begin{remark}
It is worth noting that by combining Corollary \ref{twodim} and Corollary \ref{freeproductresult} we can create an infinite family of level $\s$-lecture hall polytopes of arbitrary dimension. In particular, $\Pbn$ is level when $\s$ is any sequence satisfying $s_i=1$ when $i=0\bmod 3$.
\end{remark}

\section{Concluding remarks and future directions}\label{sec:Concluding}
There are two immediate avenues to continue this work, namely providing a more geometric classification of the Gorenstein property in the case $\gcd(s_{i-1},s_{i})\geq 2$ for all $i$ and using the levelness characterization to produce more tractable results in special cases.

With regards to the Gorenstein characterization, extensive computational evidence --- using the Normaliz software \cite{Normaliz} --- suggests that $\gcd(s_i,s_{i+1})=1$ may not be necessary.
We have the following conjecture:

\begin{conjecture}\label{GorConjecture}
Let $\s\in \Z_{\geq 1}^n$ be any sequence. The polytope $\Pbn$ is Gorenstein if and only if for all $j \geq 2$
\begin{equation*}
\rem{c_j \coloneqq }\frac{\gcd(s_{j-1},s_j)}{s_{j-1}} + \sum_{k=1}^{j-1} \frac{\gcd(s_{k-1},s_k)}{s_{k-1}s_k} \quad \text{and } \quad \rem{d_j \coloneqq }\frac{\gcd(s_{n-j+2},s_j)}{s_{n-j+1}} + \sum_{k=1}^{j-1} \frac{\gcd(s_{n-k+2},s_{n-k+1})}{s_{n-k+2}s_{n-k+1}}
\end{equation*}
are integers where $s_0 = s_{n+1}=1$.
\rem{Let $\s=(s_1,s_2,\ldots,s_n)$ be any sequence.  Then $\Pbn$ is Gorenstein if and only if there exist $\c,\d\in\Z^n$ satisfying
	\[
	c_js_{j-1}=c_{j-1}s_j+\gcd(s_j,s_{j+1})
	\]
and 
	\[
	d_j\overleftarrow{s_{j-1}}=d_{j-1}\overleftarrow{s_j}+\gcd(\overleftarrow{s_j},\overleftarrow{s_{j+1}})
	\]	
for $j>1$ with $c_1=d_1=1$.
}	
\end{conjecture}

Unfortunately, the condition $\gcd(s_{i-1},s_{i})=1$ for some $i$ is necessary for our current method of proof.
It is worth noting that examples of Gorenstein $\Pbn$ with the property that  $\gcd(s_{i-1},s_{i})\geq 2$ seem to be rare. In fact, most examples are well structured so that reductions can be made to utilize the Theorem~\ref{GorensteinChar}. For example, the sequence $\s=(2,4,\ldots,2n)$ produces a Gorenstein polytope and satisfies the condition of Conjecture~\ref{GorConjecture}. However, we can also realize $\Pb_n^{(2,4,\ldots,2n)}=2\cdot \Pb_n^{(1,2,\ldots,n)}$, and $\Pb_n^{(1,2,\ldots,n)}$ is Gorenstein by the classification and it is easy to see that $h^\ast_{n-1}(\Pb_{n}^{(1,2,\ldots,n)})\neq 0$ and $h^\ast_n(\Pb_n^{(1,2,\ldots,n)})=0$. These conditions together with results in \cite{DeNegriHibi} all imply that $\Pb_n^{(2,4,\ldots,2n)}$ must be a Gorenstein polytope as well.
In fact, we have not found an example of a Gorenstein $\Pbn$ with $\gcd(s_{i-1},s_{i})\geq 2$ that cannot alternatively be shown to be Gorenstein in a similar way.

Using the levelness characterization to produce more tractable results in special cases may prove fruitful.
Based on experimental evidence, we have the following conjecture for levelness in a large family of $\s$-lecture hall polytopes:

\begin{conjecture}
Let $\s\in\Zpos^n$ be a sequence such that there exists some $\c\in\Z^n$ satisfying
	\[
	c_js_{j-1}=c_{j-1}s_j +\gcd(s_{j-1},s_j)
	\]
for $j>1$ with $c_1=1$. Then $\Pbn$ is level.
\end{conjecture}
This conjecture, if true, implies that  $\Cc_n^{(\s)}$ a Gorenstein cone is sufficient for $\Pbn$ to be level. 
However, it should be noted that the characterization, though more efficient than explicitly computing the generators of the canonical module, can often be unwieldy for complicated computations.
It may, in fact, be more effective to produce an alternative representation of the level property, perhaps in terms of local cohomology. 

An additional future direction is to consider levelness in $\s$-lecture hall cones. 
There is no canonical choice of grading for the $\s$-lecture hall cones  as there is in the polytopes and the different gradings have different computational advantages (see \cite{BeckEtAl-GorensteinLHC,Olsen-HilbertBases}).
One must choose a grading before approaching this problem.
Preliminary computations with respect to certain gradings suggests that (non-Gorenstein) level $\s$-lecture hall cones are quite rare.

\subsection*{Acknowledgements}
The authors would like to thank Matthias Beck, Benjamin Braun, Christian Haase, Carla Savage, Liam Solus, and Zafeirakis Zafeirakopoulos for helpful comments and suggestions on this work.
The authors would also like to thank Takayuki Hibi and Akiyoshi Tsuchiya for organizing the Workshop on Convex Polytopes for Graduate Students at Osaka University in January 2017, which is where this work began. Moreover, the authors would like to express their gratitude to the anonymous referee for their detailed report. Their meticulous reading and constructive criticism improved this paper significantly.

\newcommand{\etalchar}[1]{$^{#1}$}
\providecommand{\bysame}{\leavevmode\hbox to3em{\hrulefill}\thinspace}
\providecommand{\MR}{\relax\ifhmode\unskip\space\fi MR }
\providecommand{\MRhref}[2]{%
  \href{http://www.ams.org/mathscinet-getitem?mr=#1}{#2}
}
\providecommand{\href}[2]{#2}

\rem{
\bibliographystyle{amsalpha}
\bibliography{bib_LALHP}{}

}

\end{document}